\documentclass{amsart}
\usepackage{setspace}
\usepackage{a4}
\usepackage{amsthm,latexsym}
\usepackage{amsfonts}
\usepackage{graphicx}
\usepackage{textcomp}
\usepackage{cite}
\usepackage{enumerate}
\usepackage[mathscr]{euscript}
\usepackage{mathtools}
\usepackage{amssymb}
\usepackage{hyperref}
\usepackage{amsmath}
\usepackage{tikz}
\newtheorem{theorem}{Theorem}[section]

\newtheorem{conjecture}[theorem]{Conjecture}

\newtheorem{definition}[theorem]{Definition}

\newtheorem{question}[theorem]{Question}
\title{This is the title}

\usepackage{fancyhdr}

\begin{document}
\hrule\hrule\hrule\hrule\hrule
\vspace{0.3cm}	
\begin{center}
{\bf{NON-ARCHIMEDEAN AND p-ADIC FUNCTIONAL WELCH BOUNDS}}\\
\vspace{0.3cm}
\hrule\hrule\hrule\hrule\hrule
\vspace{0.3cm}
\textbf{K. MAHESH KRISHNA}\\
Post Doctoral Fellow \\
Statistics and Mathematics Unit\\
Indian Statistical Institute, Bangalore Centre\\
Karnataka 560 059, India\\
Email: kmaheshak@gmail.com\\

Date: \today
\end{center}

\hrule\hrule
\vspace{0.5cm}
\textbf{Abstract}: We prove the non-Archimedean (resp. p-adic) Banach space version of non-Archimedean  (resp. p-adic) Welch bounds recently obtained by M. Krishna. More precisely, we prove following results.  
\begin{enumerate}[\upshape(i)]
	\item Let $\mathbb{K}$ be a non-Archimedean (complete) valued field satisfying $\left|\sum_{j=1}^{n}\lambda_j^2\right|=\max_{1\leq j \leq n}|\lambda_j|^2$ for all $ \lambda_j \in \mathbb{K}, 1\leq j \leq n$, for all $n \in \mathbb{N}.$ Let $\mathcal{X}$ be a $d$-dimensional non-Archimedean Banach space over $\mathbb{K}$. If  $\{\tau_j\}_{j=1}^n$ is any  collection in $\mathcal{X}$ and $\{f_j\}_{j=1}^n$ is any  collection in $\mathcal{X}^*$ (dual of $\mathcal{X}$)
satisfying $f_j(\tau_j) =1$ for all $1\leq j \leq n$ and   the operator $S_{f, \tau} : \text{Sym}^m(\mathcal{X})\ni x \mapsto \sum_{j=1}^nf_j^{\otimes m}(x)\tau_j^{\otimes m} \in \text{Sym}^m(\mathcal{X})$, is diagonalizable, then 
\begin{align}\label{NONFUNCTIONALWELCH}
\max_{1\leq j,k \leq n, j \neq k}\{|n|, |f_j(\tau_k)f_k(\tau_j)|^{m} \}\geq \frac{|n|^2}{\left|{d+m-1 \choose m}\right| }.	
\end{align}
We call Inequality  (\ref{NONFUNCTIONALWELCH})   as   non-Archimedean functional Welch bounds.
\item For a prime $p$, let $\mathbb{Q}_p$ be the p-adic number field. Let $\mathcal{X}$ be a $d$-dimensional p-adic  Banach space over $\mathbb{Q}_p$. If  $\{\tau_j\}_{j=1}^n$ is any  collection in $\mathcal{X}$ and $\{f_j\}_{j=1}^n$ is any  collection in $\mathcal{X}^*$ (dual of $\mathcal{X}$) satisfying $f_j(\tau_j) =1$ for all $1\leq j \leq n$ and   there exists $b \in \mathbb{Q}_p$ such that $ \sum_{j=1}^{n}f_j^{\otimes m}(x) \tau_j^{\otimes m} =bx$ for all $ x \in \text{Sym}^m(\mathcal{X}),$ then 
\begin{align}\label{PADICFUNCTIONALWELCH}
\max_{1\leq j,k \leq n, j \neq k}\{|n|, |f_j(\tau_k)f_k(\tau_j)|^{m} \}\geq \frac{|n|^2}{\left|{d+m-1 \choose m}\right| }.	
\end{align}
We call Inequality  (\ref{PADICFUNCTIONALWELCH})   as   p-adic  functional Welch bounds.
\end{enumerate}
We formulate non-Archimedean functional and p-adic functional  Zauner conjectures.



\hrule

\section{Introduction}
Everything starts from the result  Prof. L. Welch, obtained in 1974 \cite{WELCH}.
  \begin{theorem}\cite{WELCH}\label{WELCHTHEOREM} (\textbf{Welch Bounds})
	Let $n> d$.	If	$\{\tau_j\}_{j=1}^n$  is any collection of  unit vectors in $\mathbb{C}^d$, then
	\begin{align*}
		\sum_{1\leq j,k \leq n}|\langle \tau_j, \tau_k\rangle |^{2m}=\sum_{j=1}^n\sum_{k=1}^n|\langle \tau_j, \tau_k\rangle |^{2m}\geq \frac{n^2}{{d+m-1\choose m}}, \quad \forall m \in \mathbb{N}.
	\end{align*}
	In particular,
	\begin{align*}
	\sum_{1\leq j,k \leq n}|\langle \tau_j, \tau_k\rangle |^{2}=		\sum_{j=1}^n\sum_{k=1}^n|\langle \tau_j, \tau_k\rangle |^{2}\geq \frac{n^2}{{d}}.
	\end{align*}
	Further, 
	\begin{align*}
		\text{(\textbf{Higher order Welch bounds})}	\quad		\max _{1\leq j,k \leq n, j\neq k}|\langle \tau_j, \tau_k\rangle |^{2m}\geq \frac{1}{n-1}\left[\frac{n}{{d+m-1\choose m}}-1\right], \quad \forall m \in \mathbb{N}.
	\end{align*}
	In particular,
	\begin{align*}
		\text{(\textbf{First order Welch bound})}\quad 	\max _{1\leq j,k \leq n, j\neq k}|\langle \tau_j, \tau_k\rangle |^{2}\geq\frac{n-d}{d(n-1)}.
	\end{align*}
\end{theorem}
  Theorem \ref{WELCHTHEOREM} is a powerful tool in many areas such as     in the study of root-mean-square (RMS) absolute cross relation of unit vectors   \cite{SARWATEMEETING}, frame potential \cite{BENEDETTOFICKUS, CASAZZAFICKUSOTHERS, BODMANNHAASPOTENTIAL}, 
 correlations \cite{SARWATE},  codebooks \cite{DINGFENG}, numerical search algorithms  \cite{XIA, XIACORRECTION}, quantum measurements 
\cite{SCOTTTIGHT}, coding and communications \cite{TROPPDHILLON, STROHMERHEATH}, code division multiple access (CDMA) systems \cite{CHEBIRA1, CHEBIRA2}, wireless systems \cite{YATES}, compressed/compressive sensing \cite{TAN, VIDYASAGAR, FOUCARTRAUHUT, ELDARKUTYNIOK, BAJWACALDERBANKMIXON, TROPP, SCHNASSVANDERGHEYNST, ALLTOP},  `game of Sloanes' \cite{JASPERKINGMIXON}, equiangular tight frames \cite{SUSTIKTROPP}, equiangular lines \cite{MIXONSOLAZZO, COUTINHOGODSILSHIRAZIZHAN, FICKUSJASPERMIXON, IVERSONMIXON2022}, digital fingerprinting \cite{MIXONQUINNKIYAVASHFICKUS}  etc.

 Theorem \ref{WELCHTHEOREM}  has been upgraded/different proofs were  given   in \cite{CHRISTENSENDATTAKIM, DATTAWELCHLMA, WALDRONSH, WALDRON2003, DATTAHOWARD, ROSENFELD, HAIKINZAMIRGAVISH, EHLEROKOUDJOU, STROHMERHEATH}.  In 2021 M. Krishna derived continuous version of  Theorem \ref{WELCHTHEOREM} \cite{MAHESHKRISHNA}. In 2022 M. Krishna obtained Theorem \ref{WELCHTHEOREM} for Hilbert C*-modules \cite{MAHESHKRISHNA2}, Banach spaces \cite{MAHESHKRISHNA3},  non-Archimedean Hilbert spaces  \cite{MAHESHKRISHNA4} and p-adic Hilbert spaces \cite{MAHESHKRISHNA5}.
 
In this paper we derive non-Archimedean (resp. p-adic) Banach space  version  of non-Archimedean (resp. p-adic) Welch bounds in Theorem \ref{WELCHNON11F} (resp. Theorem  \ref{WELCHNON2}). We formulate  non-Archimedean functional  Zauner conjecture (Conjecture \ref{NZF}) and p-adic functional Zauner conjecture (Conjecture \ref{NZ}). We also formulate  non-Archimedean  functional  equiangular  line problem (Question  \ref{NZF}) and p-adic functional equiangular  line problem (Question \ref{NZ}).

 \section{Non-Archimedean Functional Welch bounds}
 In this section we derive non-Archimedean Banach space version of results derived  in \cite{MAHESHKRISHNA4}. Let $\mathbb{K}$ be a non-Archimedean (complete) valued field satisfying 
 \begin{align}\label{FU}
 	\left|\sum_{j=1}^{n}\lambda_j^2\right|=\max_{1\leq j \leq n}|\lambda_j|^2, \quad \forall \lambda_j \in \mathbb{K}, 1\leq j \leq n, \forall n \in \mathbb{N}.
 \end{align}
 For examples of such fields, we refer   \cite{PEREZGARCIASCHIKHOF}. Throughout this section,  we assume that our non-Archimedean field satisfies Equation (\ref{FU}). Letter $\mathcal{X}$ stands for a  $d$-dimensional non-Archimedean Banach space over $\mathbb{K}$. Identity operator on $\mathcal{X}$ is denoted by $I_\mathcal{X}$. The dual of $\mathcal{X}$ is denoted by $\mathcal{X}^*$. 
  \begin{theorem}\label{WELCHNON1}
 	\textbf{(First Order  Non-Archimedean Functional Welch Bound)}	Let $\mathcal{X}$ be a $d$-dimensional non-Archimedean Banach space over $\mathbb{K}$.	  If $\{\tau_j\}_{j=1}^n$ is any  collection in $\mathcal{X}$ and $\{f_j\}_{j=1}^n$ is any  collection in $\mathcal{X}^*$ 
 	such that  the operator $S_{f, \tau} : \mathcal{X}\ni x \mapsto \sum_{j=1}^nf_j(x) \tau_j \in \mathcal{X}$ 
 	is diagonalizable, then 
 	\begin{align*}
 		\max_{1\leq j,k \leq n, j \neq k}\left \{\left| \sum_{l=1}^nf_l(\tau_l)^2 \right|, |f_j(\tau_k)f_k(\tau_j)|\right\}\geq \frac{1}{|d|}	\left|\sum_{j=1}^nf_j(\tau_j) \right|^2.	
 	\end{align*}
 	In particular, if $f_j(\tau_j) =1$ for all $1\leq j \leq n$, then 
 	\begin{align*}
 		\text{\textbf{(First order  non-Archimedean functional Welch bound)}} \quad \max_{1\leq j,k \leq n, j \neq k}\{|n|, |f_j(\tau_k)f_k(\tau_j)| \}\geq \frac{|n|^2}{|d|}.	
 	\end{align*}
 \end{theorem}
 \begin{proof}
 	We first note that 
 	\begin{align*}
 		&\operatorname{Tra}(S_{f, \tau})=\sum_{j=1}^nf_j(\tau_j) , \\
 		& \operatorname{Tra}(S^2_{f, \tau})=\sum_{j=1}^n\sum_{k=1}^nf_j(\tau_k)f_k(\tau_j).
 	\end{align*}	
 	Let $\lambda_1, \dots, \lambda_d$ be the diagonal entries in the diagonalization of   $S_{f, \tau}$. Then  using the diagonalizability of $	S_{f, \tau}$ and the non-Archimedean  Cauchy-Schwarz inequality (Theorem 2.4.2 \cite{PEREZGARCIASCHIKHOF}), we get 
 		\begin{align*}
 		\left|\sum_{j=1}^nf_j(\tau_j) \right|^2&=|\operatorname{Tra}(S_{f, \tau})|^2=\left|\sum_{k=1}^d
 		\lambda_k\right|^2\leq |d| 	\left|\sum_{k=1}^d\lambda_k^2 \right|=|d||\operatorname{Tra}(S^2_{f, \tau})|\\
 		&=|d|\left|\sum_{j=1}^n\sum_{k=1}^nf_j(\tau_k)f_k(\tau_j)\right|=|d|\left| \sum_{l=1}^nf_l(\tau_l)^2+\sum_{j,k=1, j\neq k}^nf_j(\tau_k)f_k(\tau_j)\right|\\
 		&\leq |d| \max_{1\leq j,k \leq n, j \neq k}\left \{\left| \sum_{l=1}^nf_l(\tau_l) ^2\right|, |f_j(\tau_k)f_k(\tau_j)| \right\}.
 	\end{align*}
 	Whenever $f_j(\tau_j) =1$ for all $1\leq j \leq n$, 
 	\begin{align*}
 		|n|^2\leq |d|\max_{1\leq j,k \leq n, j \neq k}\{|n|, |f_j(\tau_k)f_k(\tau_j)| \}.
 	\end{align*}
 \end{proof}
 Next we obtain higher order non-Archimedean functional Welch bounds. We use the following vector space result.
 \begin{theorem}\cite{COMON, BOCCI}\label{SYMMETRICTENSORDIMENSION}
 	If $\mathcal{V}$ is a vector space of dimension $d$ and $\text{Sym}^m(\mathcal{V})$ denotes the vector space of symmetric m-tensors, then 
 	\begin{align*}
 		\text{dim}(\text{Sym}^m(\mathcal{V}))={d+m-1 \choose m}, \quad \forall m \in \mathbb{N}.
 	\end{align*}
 \end{theorem}
 
 \begin{theorem}\label{WELCHNON11F}
 	(\textbf{Higher Order Non-Archimedean Functional Welch Bounds}) Let $\mathcal{X}$ be a $d$-dimensional non-Archimedean Banach space over $\mathbb{K}$. Let $m \in \mathbb{N}$.   If $\{\tau_j\}_{j=1}^n$ is any  collection in $\mathcal{X}$ and $\{f_j\}_{j=1}^n$ is any  collection in $\mathcal{X}^*$ such that  the operator $S_{f, \tau} : \text{Sym}^m(\mathcal{X})\ni x \mapsto \sum_{j=1}^nf_j^{\otimes m}(x)\tau_j^{\otimes m} \in \text{Sym}^m(\mathcal{X})$ 
 	is diagonalizable, then 
 	\begin{align*}
 		\max_{1\leq j,k \leq n, j \neq k}\left \{\left| \sum_{l=1}^nf_l(\tau_l)^{2m} \right|, |f_j(\tau_k)f_k(\tau_j)|^{m}\right\}\geq \frac{1}{\left|{d+m-1 \choose m}\right|}\left|\sum_{j=1}^nf_j(\tau_j)^m \right|^2.	
 	\end{align*}
 	In particular, if $f_j(\tau_j) =1$ for all $1\leq j \leq n$, then
 	\begin{align*}
 		\text{\textbf{(Higher order  non-Archimedean functional Welch bounds)}} \\
 		\quad  \max_{1\leq j,k \leq n, j \neq k}\{|n|, |f_j(\tau_k)f_k(\tau_j)|^{m} \}\geq \frac{|n|^2}{\left|{d+m-1 \choose m}\right| }.	
 	\end{align*}
 \end{theorem}
 \begin{proof}
 	Let $\lambda_1, \dots, \lambda_{\text{dim}(\text{Sym}^m(\mathcal{X}))}$ be the diagonal entries in the diagonalization of  $S_{f, \tau}$. We note that 
 	 \begin{align*}
 		&b\operatorname{dim(\text{Sym}^m(\mathcal{X}))}=\operatorname{Tra}(bI_{\text{Sym}^m(\mathcal{X})})=\operatorname{Tra}(S_{f, \tau})=\sum_{j=1}^nf_j^{\otimes m}(\tau_j^{\otimes m}), \\
 		& b^2\operatorname{dim(\text{Sym}^m(\mathcal{X}))}=\operatorname{Tra}(b^2I_{\text{Sym}^m(\mathcal{X})})=\operatorname{Tra}(S^2_{f, \tau})=\sum_{j=1}^n\sum_{k=1}^nf_j^{\otimes m} (\tau^{\otimes m} _k) f^{\otimes m} _k(\tau^{\otimes m} _j).
 	\end{align*}	
 	 Then 
 		\begin{align*}
 		&\left|\sum_{j=1}^nf_j(\tau_j)^m \right|^2=	\left|\sum_{j=1}^nf_j^{\otimes m}(\tau_j^{\otimes m}) \right|^2=|\operatorname{Tra}(S_{f, \tau})|^2=\left|\sum_{k=1}^{\text{dim}(\text{Sym}^m(\mathcal{X}))}
 		\lambda_k\right|^2\\
 		&\leq |\text{dim}(\text{Sym}^m(\mathcal{X}))| 	\left|\sum_{k=1}^{\text{dim}(\text{Sym}^m(\mathcal{X}))}\lambda_k^2 \right|= |\text{dim}(\text{Sym}^m(\mathcal{X}))| |\operatorname{Tra}(S^2_{f, \tau})|\\
 		&=\left|{d+m-1 \choose m}\right||\operatorname{Tra}(S^2_{f,\tau})|=\left|{d+m-1 \choose m}\right|\left|\sum_{j=1}^n\sum_{k=1}^nf_j^{\otimes m}(\tau_k^{\otimes m})  f_k^{\otimes m} (\tau_j^{\otimes m}) \right|\\
 		&=\left|{d+m-1 \choose m}\right|\left|\sum_{j=1}^n\sum_{k=1}^nf_j(\tau_k) ^mf_k(\tau_j)^m\right|\\
 		&=\left|{d+m-1 \choose m}\right|\left| \sum_{l=1}^nf_l(\tau_l)^{2m}+\sum_{j,k=1, j\neq k}^nf_j(\tau_k) ^mf_k(\tau_j)^m\right|\\
 		&\leq \left|{d+m-1 \choose m}\right| \max_{1\leq j,k \leq n, j \neq k}\left \{\left| \sum_{l=1}^nf_l(\tau_l)^{2m} \right|, |f_j(\tau_k) ^mf_k(\tau_j)^m| \right\}\\
 		&=\left|{d+m-1 \choose m}\right| \max_{1\leq j,k \leq n, j \neq k}\left \{\left| \sum_{l=1}^nf_l(\tau_l)^{2m} \right|, |f_j(\tau_k)f_k(\tau_j)|^m \right\}.
 	\end{align*}
 	Whenever $f_j(\tau_j)=1$ for all $1\leq j \leq n$, 
 	\begin{align*}
 		|n|^2\leq  \left|{d+m-1 \choose m}\right| \max_{1\leq j,k \leq n, j \neq k}\{|n|, |f_j(\tau_k)f_k(\tau_j)|^{m} \}.
 	\end{align*}
 \end{proof}
Motivated from  \ref{WELCHNON1} we formulate the following  question.
 \begin{question}\label{Q1F}
 	\textbf{Let $\mathbb{K}$ non-Archimedean field  satisfying Equation (\ref{FU}) and $\mathcal{X}$ be a $d$-dimensional non-Archimedean Banach space over $\mathbb{K}$.
 		For which  $n	\in  \mathbb{N}$,  there exist vectors $\tau_1, \dots, \tau_n \in \mathcal{X}$ and functionals $f_1, \dots, f_n \in \mathcal{X}^*$      satisfying the following.
 		\begin{enumerate}[\upshape(i)]
 			\item $f_j(\tau_j) =1$ for all $1\leq j \leq n$.
 			\item The operator $S_{f, \tau} : \mathcal{X}\ni x \mapsto \sum_{j=1}^nf_j(x) \tau_j \in \mathcal{X}$  is diagonalizable. 
 			\item 
 			\begin{align*}
 				\max_{1\leq j,k \leq n, j \neq k}\{|n|, |f_j(\tau_k)f_k(\tau_j)| \}= \frac{|n|^2}{|d|}.
 			\end{align*}
 		\item $\|f_j\| =1$ for all $1\leq j \leq n$, $\|\tau_j\| =1$ for all $1\leq j \leq n$.
 	\end{enumerate}}
 \end{question}
 A particular case of Question \ref{Q1F} is the following non-Archimedean functional version of Zauner conjecture which comes by taking $n=d^2$ (see \cite{APPLEBY123, APPLEBY, ZAUNER, SCOTTGRASSL, FUCHSHOANGSTACEY, RENESBLUMEKOHOUTSCOTTCAVES, APPLEBYSYMM, BENGTSSON, APPLEBYFLAMMIAMCCONNELLYARD, KOPPCON, GOURKALEV, BENGTSSONZYCZKOWSKI, PAWELRUDNICKIZYCZKOWSKI, BENGTSSON123, MAGSINO, MAHESHKRISHNA} for Zauner conjecture in Hilbert spaces,   \cite{MAHESHKRISHNA2}  for Zauner conjecture in Hilbert C*-modules,     \cite{MAHESHKRISHNA3}  for Zauner conjecture in Banach spaces,  
 \cite{MAHESHKRISHNA4}	 for Zauner conjecture in non-Archimedean Hilbert spaces and  \cite{MAHESHKRISHNA5} for Zauner conjecture in p-adic Hilbert spaces).
 \begin{conjecture}\label{NZF} \textbf{(Non-Archimedean Functional Zauner Conjecture)
 		Let $\mathbb{K}$ non-Archimedean field  satisfying Equation (\ref{FU}). 	For each $d\in \mathbb{N}$, there exist vectors $\tau_1, \dots, \tau_{d^2} \in \mathbb{K}^d$ (w.r.t. any non-Archimedean norm)  and functionals $f_1, \dots, f_{d^2} \in ( \mathbb{K}^d)^*$ satisfying the following.
 		\begin{enumerate}[\upshape(i)]
 			\item $f_j(\tau_j) =1$ for all $1\leq j \leq d^2$.
 			\item The operator $S_{f, \tau} : \mathbb{K}^d\ni x \mapsto \sum_{j=1}^{d^2}f_j(x) \tau_j \in \mathbb{K}^d$  is diagonalizable. 
 			\item 
 			\begin{align*}
 				|f_j(\tau_k)f_k(\tau_j)| =|d|, \quad \forall 1\leq j, k \leq d^2, j \neq k.
 			\end{align*}
 	\item $\|f_j\| =1$ for all $1\leq j \leq d^2$, $\|\tau_j\| =1$ for all $1\leq j \leq d^2$.
 \end{enumerate}}
 \end{conjecture}
 There are four bounds which are companions of   Welch bounds in  Hilbert spaces. To recall them we need the notion of Gerzon's bound.
 \begin{definition}\cite{JASPERKINGMIXON}
 	Given $d\in \mathbb{N}$, define \textbf{Gerzon's bound}
 	\begin{align*}
 		\mathcal{Z}(d, \mathbb{K})\coloneqq 
 		\left\{ \begin{array}{cc} 
 			d^2 & \quad \text{if} \quad \mathbb{K} =\mathbb{C}\\
 			\frac{d(d+1)}{2} & \quad \text{if} \quad \mathbb{K} =\mathbb{R}.\\
 		\end{array} \right.
 	\end{align*}	
 \end{definition}
 \begin{theorem}\cite{JASPERKINGMIXON, XIACORRECTION, MUKKAVILLISABHAWALERKIPAAZHANG, SOLTANALIAN, BUKHCOX, CONWAYHARDINSLOANE, HAASHAMMENMIXON, RANKIN}  \label{LEVENSTEINBOUND}
 	Define $\mathbb{K}=\mathbb{R}$ or $\mathbb{C}$ and    $m\coloneqq \operatorname{dim}_{\mathbb{R}}(\mathbb{K})/2$.	If	$\{\tau_j\}_{j=1}^n$  is any collection of  unit vectors in $\mathbb{K}^d$, then
 	\begin{enumerate}[\upshape(i)]
 		\item (\textbf{Bukh-Cox bound})
 		\begin{align*}
 			\max _{1\leq j,k \leq n, j\neq k}|\langle \tau_j, \tau_k\rangle |\geq \frac{\mathcal{Z}(n-d, \mathbb{K})}{n(1+m(n-d-1)\sqrt{m^{-1}+n-d})-\mathcal{Z}(n-d, \mathbb{K})}\quad \text{if} \quad n>d.
 		\end{align*}
 		\item (\textbf{Orthoplex/Rankin bound})	
 		\begin{align*}
 			\max _{1\leq j,k \leq n, j\neq k}|\langle \tau_j, \tau_k\rangle |\geq\frac{1}{\sqrt{d}} \quad \text{if} \quad n>\mathcal{Z}(d, \mathbb{K}).
 		\end{align*}
 		\item (\textbf{Levenstein bound})	
 		\begin{align*}
 			\max _{1\leq j,k \leq n, j\neq k}|\langle \tau_j, \tau_k\rangle |\geq \sqrt{\frac{n(m+1)-d(md+1)}{(n-d)(md+1)}} \quad \text{if} \quad n>\mathcal{Z}(d, \mathbb{K}).
 		\end{align*}
 		\item (\textbf{Exponential bound})
 		\begin{align*}
 			\max _{1\leq j,k \leq n, j\neq k}|\langle \tau_j, \tau_k\rangle |\geq 1-2n^{\frac{-1}{d-1}}.
 		\end{align*}
 	\end{enumerate}	
 \end{theorem}
 Motivated from Theorem \ref{LEVENSTEINBOUND}  and Theorem \ref{WELCHNON1}  we ask  the following problem.
 \begin{question}
 	\textbf{Whether there is a   non-Archimedean functional version of Theorem \ref{LEVENSTEINBOUND}? In particular, does there exists a   version of}
 	\begin{enumerate}[\upshape(i)]
 		\item \textbf{non-Archimedean functional Bukh-Cox bound?}
 		\item \textbf{non-Archimedean functional Orthoplex/Rankin bound?}
 		\item \textbf{non-Archimedean functional Levenstein bound?}
 		\item \textbf{non-Archimedean functional Exponential bound?}
 	\end{enumerate}		
 \end{question}
 In the introduction we  wrote that  Welch bounds have applications in study of equiangular lines.  Therefore  wish to formulate equiangular line  problem for non-Archimedean Banach spaces. For the study of equiangular lines in Hilbert spaces we refer \cite{LEMMENSSEIDEL, JIANGTIDORYAOZHAOZHAO, GREAVESKOOLENMUNEMASASZOLLOSI, BALLADRAXLERKEEVASHSUDAKOV, BUKH, DECAEN, GLAZYRINYU, BARGYU, JIANGPOLYANSKII, NEUMAIER, GREAVESSYATRIADIYATSYNA, GODSILROY, CALDERBANKCAMERONKANTORSEIDEL, OKUDAYU, YU}, quaternion Hilbert space we refer \cite{ETTAOUI}, octonion Hilbert space we refer \cite{COHNKUMARMINTON}, finite dimensional vector spaces over finite fields we refer \cite{GREAVESIVERSONJASPERMIXON1, GREAVESIVERSONJASPERMIXON2},  for Banach spaces we refer \cite{MAHESHKRISHNA3},  for non-Archimedean Hilbert spaces we refer \cite{MAHESHKRISHNA4} and for p-adic Hilbert spaces we refer \cite{MAHESHKRISHNA5}.
 \begin{question}\label{EQUIF}
 	\textbf{(Non-Archimedean Functional Equiangular Line Problem)	Let $\mathbb{K}$ be a non-Archimedean field satisfying Equation (\ref{FU}). Given $a\in \mathbb{K}$, $d \in \mathbb{N}$ and $\gamma>0$, what is the maximum $n =n(\mathbb{K}, a,d, \gamma)\in \mathbb{N}$ such that there exist vectors $\tau_1, \dots, \tau_n \in \mathbb{K}^d$ (w.r.t. any non-Archimedean norm)  and functionals $f_1,\dots, f_n \in (\mathbb{K}^d)^*$ satisfying the following.
 		\begin{enumerate}[\upshape(i)]
 			\item $f_j(\tau_j) =a$ for all $1\leq j \leq n$.
 			\item $|f_j(\tau_k)f_k(\tau_j)| =\gamma$ for all $1\leq j, k \leq n, j \neq k$.
 	    	\item $\|f_j\| =1$ for all $1\leq j \leq n$, $\|\tau_j\| =1$ for all $1\leq j \leq n$.
 		\end{enumerate}
 		In particular, whether there is a non-Archimedean functional Gerzon bound?}
 \end{question}
 Question \ref{EQUIF} can be easily generalized  to  formulate question of   non-Archimedean functional  regular $s$-distance sets.

\section{p-adic Functional Welch bounds}
In this section we derive p-adic Banach space version of results done in \cite{MAHESHKRISHNA5}. Let p be a prime and $\mathbb{Q}_p$ be the filed of p-adic numbers. In this section, $\mathcal{X}$ is a $d$-dimensional p-adic Banach space over $\mathbb{Q}_p$. 
\begin{theorem}\label{WELCHNONQ}
\textbf{(First Order  p-adic Functional Welch Bound)}	Let $p$ be a prime and $\mathcal{X}$ be a $d$-dimensional p-adic Banach space over $\mathbb{Q}_p$. If 	   $\{\tau_j\}_{j=1}^n$ is any  collection in $\mathcal{X}$ and $\{f_j\}_{j=1}^n$ is any  collection in $\mathcal{X}^*$ 
such that  there exists $b \in \mathbb{Q}_p$ satisfying 
\begin{align*}
	\sum_{j=1}^{n}f_j(x)\tau_j =bx, \quad \forall x \in \mathcal{X},
\end{align*} 
then 
\begin{align*}
 \max_{1\leq j,k \leq n, j \neq k}\left \{\left| \sum_{l=1}^nf_l(\tau_l)^2 \right|, |f_j(\tau_k)f_k(\tau_j)|\right\}\geq \frac{1}{|d|}	\left|\sum_{j=1}^nf_j(\tau_j) \right|^2.	
\end{align*}
In particular, if $f_j(\tau_j) =1$ for all $1\leq j \leq n$, then 
\begin{align*}
 \text{\textbf{(First order  p-adic  functional Welch bound)}} \quad \max_{1\leq j,k \leq n, j \neq k}\{|n|, |f_j(\tau_k)f_k(\tau_j)| \}\geq \frac{|n|^2}{|d|}.	
\end{align*}
\end{theorem}
\begin{proof}
	Define $S_{f, \tau} : \mathcal{X}\ni x \mapsto \sum_{j=1}^nf_j(x) \tau_j \in \mathcal{X}$. Then 
\begin{align*}
	&bd=\operatorname{Tra}(bI_{\mathcal{X}})=\operatorname{Tra}(S_{f, \tau})=\sum_{j=1}^nf_j(\tau_j), \\
	& b^2d=\operatorname{Tra}(b^2I_{\mathcal{X}})=\operatorname{Tra}(S^2_{f,\tau})=\sum_{j=1}^n\sum_{k=1}^nf_j(\tau_k)f_k(\tau_j).
\end{align*}	
Therefore 

\begin{align*}
	\left|\sum_{j=1}^nf_j(\tau_j) \right|^2&=|\operatorname{Tra}(S_{f, \tau})|^2=|bd|^2=|d||b^2d|	
	=|d|\left|\sum_{j=1}^n\sum_{k=1}^nf_j(\tau_k)f_k(\tau_j)\right|\\
	&=|d|\left| \sum_{l=1}^nf_l(\tau_l)^2+\sum_{j,k=1, j\neq k}^nf_j(\tau_k)f_k(\tau_j)\right|\\
&\leq |d| \max_{1\leq j,k \leq n, j \neq k}\left \{\left| \sum_{l=1}^nf_l(\tau_l)^2 \right|, |f_j(\tau_k)f_k(\tau_j)| \right\}.
\end{align*}
Whenever $f_j(\tau_j) =1$ for all $1\leq j \leq n$, 
\begin{align*}
		|n|^2\leq |d|\max_{1\leq j,k \leq n, j \neq k}\{|n|, |f_j(\tau_k)f_k(\tau_j)| \}.
\end{align*}		
\end{proof}
We derive higher order version of Theorem \ref{WELCHNONQ}.
  \begin{theorem}\label{WELCHNON2}
(\textbf{Higher Order p-adic  Functional Welch Bounds}) Let $p$ be a prime and $\mathcal{X}$ be a $d$-dimensional p-adic Banach space over $\mathbb{Q}_p$.
If 	   $\{\tau_j\}_{j=1}^n$ is any  collection in $\mathcal{X}$ and $\{f_j\}_{j=1}^n$ is any  collection in $\mathcal{X}^*$ 
such that  there exists $b \in \mathbb{Q}_p$ satisfying 
\begin{align*}
	\sum_{j=1}^{n}f_j^{\otimes m}(x) \tau_j^{\otimes m} =bx, \quad \forall x \in \text{Sym}^m(\mathcal{X}),
\end{align*} 
then 
\begin{align*}
\max_{1\leq j,k \leq n, j \neq k}\left \{\left| \sum_{l=1}^nf_l(\tau_l) ^{2m} \right|, |f_j(\tau_k)f_k(\tau_j)|^{m}\right\}\geq \frac{1}{\left|{d+m-1 \choose m}\right|}\left|\sum_{j=1}^nf_j(\tau_j) ^m \right|^2.	
\end{align*}
In particular, if $f_j(\tau_j) =1$ for all $1\leq j \leq n$, then
\begin{align*}
 \text{\textbf{(Higher order p-adic  functional Welch bound)}} \quad  \max_{1\leq j,k \leq n, j \neq k}\{|n|, |f_j(\tau_k)f_k(\tau_j)|^{m} \}\geq \frac{|n|^2}{\left|{d+m-1 \choose m}\right| }.	
\end{align*}
 \end{theorem}
  \begin{proof}
 Define $S_{f, \tau} : \text{Sym}^m(\mathcal{X})\ni x \mapsto \sum_{j=1}^nf_j^{\otimes m}(x)\tau_j^{\otimes m} \in \text{Sym}^m(\mathcal{X})$. Then 
    \begin{align*}
    	&b\operatorname{dim(\text{Sym}^m(\mathcal{X}))}=\operatorname{Tra}(bI_{\text{Sym}^m(\mathcal{X})})=\operatorname{Tra}(S_{f, \tau})=\sum_{j=1}^nf_j^{\otimes m}(\tau_j^{\otimes m}), \\
    	& b^2\operatorname{dim(\text{Sym}^m(\mathcal{X}))}=\operatorname{Tra}(b^2I_{\text{Sym}^m(\mathcal{X})})=\operatorname{Tra}(S^2_{f, \tau})=\sum_{j=1}^n\sum_{k=1}^nf_j^{\otimes m} (\tau^{\otimes m} _k) f^{\otimes m} _k(\tau^{\otimes m} _j).
    \end{align*}	
Therefore 

    \begin{align*}
    	&\left|\sum_{j=1}^nf_j(\tau_j) ^m \right|^2=	\left|\sum_{j=1}^nf_j^{\otimes m}(\tau_j^{\otimes m}) \right|^2=|\operatorname{Tra}(S_{f, \tau})|^2=\left|b\operatorname{dim(\text{Sym}^m(\mathcal{X}))}\right|^2\\
    	&=\left|\operatorname{dim(\text{Sym}^m(\mathcal{X}))}\right|\left|b^2\operatorname{dim(\text{Sym}^m(\mathcal{X}))}\right|\\
    	&=\left|\operatorname{dim(\text{Sym}^m(\mathcal{X}))}\right|\left|\sum_{j=1}^n\sum_{k=1}^nf_j^{\otimes m}(\tau_k^{\otimes m}) f_k^{\otimes m}( \tau_j^{\otimes m}) \right|\\
    	&=\left|{d+m-1 \choose m}\right|\left|\sum_{j=1}^n\sum_{k=1}^nf_j^{\otimes m}( \tau_k^{\otimes m}) f_k^{\otimes m}(\tau_j^{\otimes m})\right|\\
    	&=\left|{d+m-1 \choose m}\right|\left|\sum_{j=1}^n\sum_{k=1}^nf_j(\tau_k)^m f_k(\tau_j)^m\right|\\
    	&=\left|{d+m-1 \choose m}\right|\left| \sum_{l=1}^nf_l(\tau_l)^{2m}+\sum_{j,k=1, j\neq k}^nf_j(\tau_k)^m f_k(\tau_j)^m\right|\\
    	&\leq \left|{d+m-1 \choose m}\right| \max_{1\leq j,k \leq n, j \neq k}\left \{\left| \sum_{l=1}^nf_l(\tau_l)^{2m} \right|, |f_j(\tau_k)^m f_k(\tau_j)^m| \right\}\\
    	&= \left|{d+m-1 \choose m}\right| \max_{1\leq j,k \leq n, j \neq k}\left \{\left| \sum_{l=1}^nf_l(\tau_l)^{2m} \right|, |f_j(\tau_k)f_k(\tau_j)|^{m}\right\}.
    \end{align*}
    Whenever $f_j(\tau_j) =1$ for all $1\leq j \leq n$, 
    
    \begin{align*}
    	|n|^2\leq  \left|{d+m-1 \choose m}\right| \max_{1\leq j,k \leq n, j \neq k}\{|n|, |f_j(\tau_k)f_k(\tau_j)|^{m} \}.
    \end{align*}
  \end{proof}
Difference between Theorem \ref{WELCHNON11F}   and  Theorem \ref{WELCHNON2} should be clearly emphasized. Assumption in (at vector space level) Theorem \ref{WELCHNON2} is more than the assumption in  Theorem \ref{WELCHNON11F} (as any scalar times identity is already diagonal) but the field in Theorem \ref{WELCHNON11F} is much more restrictive than the field in Theorem \ref{WELCHNON2}. Theorem \ref{WELCHNON2} works on any non-Archimedean field not just $\mathbb{Q}_p$. Using Theorem \ref{WELCHNONQ} we ask the  following question.
\begin{question}\label{Q1}
	\textbf{Given a prime $p$, for which $d$-dimensional p-adic Banach space $\mathcal{X}$ over $\mathbb{Q}_p$ and  $n	\in \mathbb{N}$,  there exist vectors $\tau_1, \dots, \tau_n \in \mathcal{X}$ and functionals $f_1, \dots, f_n \in \mathcal{X}^*$ satisfying the following.
		\begin{enumerate}[\upshape(i)]
			\item $f_j(\tau_j) =1$ for all $1\leq j \leq n$.
			\item  There exists $b \in \mathbb{Q}_p$ satisfying 
			\begin{align*}
				\sum_{j=1}^{n}f_j(x) \tau_j =bx, \quad \forall x \in \mathcal{X},
			\end{align*} 
			\item 
			\begin{align*}
				\max_{1\leq j,k \leq n, j \neq k}\{|n|, |f_j(\tau_k)f_k(\tau_j)| \}= \frac{|n|^2}{|d|}.
			\end{align*}
			\item $\|f_j\| =1$ for all $1\leq j \leq n$, $\|\tau_j\| =1$ for all $1\leq j \leq n$.
	\end{enumerate}}
\end{question}
A particular case of Question \ref{Q1} is the following p-adic functional  Zauner conjecture.
\begin{conjecture}\label{NZ} \textbf{(p-adic Functional Zauner Conjecture)
		Let $p$ be a prime.	For each $d\in \mathbb{N}$, there exist vectors $\tau_1, \dots, \tau_{d^2} \in \mathbb{Q}_p^d$ (w.r.t. any non-Archimedean norm) and functionals $f_1, \dots, f_n \in (\mathbb{Q}_p^d)^*$ satisfying the following.
		\begin{enumerate}[\upshape(i)]
			\item $f_j(\tau_j) =1$ for all $1\leq j \leq d^2$.
			\item  There exists $b \in \mathbb{Q}_p$ satisfying 
			\begin{align*}
				\sum_{j=1}^{d^2}f_j(x) \tau_j =bx, \quad \forall x \in \mathcal{X},
			\end{align*} 
					\item 
			\begin{align*}
				|f_j(\tau_k)f_k(\tau_j)| =|d|, \quad \forall 1\leq j, k \leq d^2, j \neq k.
			\end{align*}
			\item $\|f_j\| =1$ for all $1\leq j \leq d^2$, $\|\tau_j\| =1$ for all $1\leq j \leq d^2$.
	\end{enumerate}}
	\end{conjecture}

Theorem \ref{LEVENSTEINBOUND}  and Theorem \ref{WELCHNONQ}  give  the following problem.
\begin{question}
	\textbf{Whether there is a   p-adic functional version of Theorem \ref{LEVENSTEINBOUND}? In particular, does there exists a   version of}
	\begin{enumerate}[\upshape(i)]
		\item \textbf{p-adic functional Bukh-Cox bound?}
		\item \textbf{p-adic functional Orthoplex/Rankin bound?}
		\item \textbf{p-adic functional Levenstein bound?}
		\item \textbf{p-adic functional Exponential bound?}
	\end{enumerate}		
\end{question}
We end by formulating p-adic functional equiangular line problem.
\begin{question}\label{EQUI}
	\textbf{(p-adic Functional Equiangular Line Problem)	Let $p$ be a prime. Given $a\in \mathbb{Q}_p$, $d \in \mathbb{N}$ and $\gamma>0$, what is the maximum $n =n(p, a,d, \gamma)\in \mathbb{N}$ such that there exist vectors $\tau_1, \dots, \tau_n \in \mathbb{Q}_p^d$ (w.r.t. any non-Archimedean norm) and functionals $f_1, \dots, f_n \in (\mathbb{Q}_p^d)^*$ satisfying the following.
\begin{enumerate}[\upshape(i)]
	\item $f_j(\tau_j) =a$ for all $1\leq j \leq n$.
	\item $|f_j(\tau_k)f_k(\tau_j)| =\gamma$ for all $1\leq j, k \leq n, j \neq k$.
		\item $\|f_j\| =1$ for all $1\leq j \leq n$, $\|\tau_j\| =1$ for all $1\leq j \leq n$.
\end{enumerate}
In particular, whether there is a p-adic  functional Gerzon bound?}
\end{question}
 Question \ref{EQUI} can be easily reformulated  to  formulate question of   p-adic functional regular $s$-distance sets.

 \bibliographystyle{plain}
 \bibliography{reference.bib}

\end{document}